\documentclass[10pt,twoside]{article}
\usepackage{mathrsfs}
\usepackage{amssymb}
\usepackage{amsmath,color}
\usepackage[all]{xy}
\usepackage{amsthm}
\usepackage{url}
\numberwithin{equation}{section}

\newtheorem{theorem}{Theorem}[section]

\newtheorem{definition}[theorem]{Definition}

\newtheorem{lemma}[theorem]{Lemma}
\newtheorem{example}[theorem]{Example}

\newtheorem{corollary}[theorem]{Corollary}
\newtheorem{theorem*}{Theorem}

\newcommand{\TTF}{\operatorname{TTF}}

\newcommand{\modu}{\operatorname{mod}}

\newcommand{\Hom}{\operatorname{Hom}}

\newcommand{\add}{\operatorname{add}}

\newcommand{\Ker}{\operatorname{Ker}}
\newcommand{\Coker}{\operatorname{Coker}}
\newcommand{\Ima}{\operatorname{Im}}

\title{ \bf Torsion Pairs in Recollements of Abelian Categories
}
\vspace{0.2cm}

\author{Xin Ma and Zhaoyong Huang\thanks{E-mail address:  maxin@smail.nju.edu.cn, huangzy@nju.edu.cn
}  \\
{\it \footnotesize  Department of Mathematics, Nanjing University, Nanjing 210093, Jiangsu Province, P. R. China}}
\date{ }
\begin{document}

\baselineskip=16pt
\maketitle

\begin{abstract}
For a recollement  $(\mathcal{A},\mathcal{B},\mathcal{C})$ of abelian categories, we show that
torsion pairs in $\mathcal{A}$ and $\mathcal{C}$ can induce torsion pairs in $\mathcal{B}$;
and the converse holds true under certain conditions.
\end{abstract}

\pagestyle{myheadings}
\markboth{\rightline {\scriptsize   Xin Ma and Zhaoyong Huang}}
         {\leftline{\scriptsize Torsion Pairs in Recollements of Abelian Categories}}

%\section{Introduction} %delete * to number this section

\section{Introduction} %delete * to number this section

Recollements of abelian categories and triangulated categories play an important role in geometry of singular spaces,
representation theory, polynomial functors theory and ring theory \cite{BBD81F,BARI07H,CEPBSL86D,CEPBSL88F,JJP65S,KNJ94G,KNJ02A,PTL88P},
where recollements are known as torsion torsion-free theories. They first appeared in the construction of the category
of perverse sheaves on a singular space \cite{BBD81F}. Recollements of abelian categories and recollements of triangulated categories
are closely related; for instance, Chen \cite{CJM13C} constructed a recollement of abelian categories from a recollement
of triangulated categories, generalizing a result of Lin and Wang \cite{LYNWMX10F}. In addition, the properties of
torsion pairs and recollements of abelian categories have been studied by Psaroudakis and Vit\'{o}ria \cite{PCVJ14R},
establishing a correspondence between recollements of abelian categories up to equivalence and certain $\TTF$-triples.

Let $(\mathcal{A},\mathcal{B},\mathcal{C})$ be a recollement of triangulated categories.
Chen \cite{CJM13C} described how to glue together cotorsion pairs (which are essentially equal to torsion pairs \cite{IOYY08M}) in $\mathcal{A}$ and $\mathcal{C}$
to obtain a cotorsion pair in $\mathcal{B}$, which is a natural generalization of a similar
result in \cite{BBD81F} on gluing together $t$-structures of $\mathcal{A}$ and $\mathcal{C}$
to obtain a $t$-structure in $\mathcal{B}$. After taking the hearts $\mathcal{A}',\mathcal{B}',\mathcal{C}'$ of the glued $t$-structures,
then $(\mathcal{A}',\mathcal{B}',\mathcal{C}')$ is a recollement of abelian categories and a construction of gluing of torsion pairs in
this recollement was given by Liu, Vit\'oria and Yang in \cite{LQHVJYD14G} (also see \cite{JD09D}).
Note that the results of Liu, Vit\'oria and Yang (see \cite[Proposition 6.5 and Lemma 6.2]{LQHVJYD14G}) depend on
the recollements of triangulated categories and the proofs there do not work in the general case.
Our aim is to realize these in a recollement of general abelian categories.

This paper is organized as follows. In Section \ref{pre}, we give some terminologies and some preliminary results.
In Section \ref{tor}, we study torsion pairs in a recollement of abelian categories. Let $(\mathcal{A},\mathcal{B},\mathcal{C})$
be a recollement of abelian categories, we obtain a torsion pair in $\mathcal{B}$ from torsion pairs in $\mathcal{A}$ and $\mathcal{C}$.
Conversely, we show that, under certain conditions, a torsion pair in $\mathcal{B}$ can induce torsion pairs in $\mathcal{A}$ and $\mathcal{C}$.

\section{Preliminaries}\label{pre}
Throughout this paper, all subcategories are full, additive and closed under isomorphisms.
\begin{definition}\label{def-2.1}
{\rm (\cite{VFTP04C}) A recollement, denoted by ($\mathcal{A},\mathcal{B},\mathcal{C}$), of abelian categories is a diagram
$$\xymatrix{\mathcal{A}\ar[rr]!R|{i_{*}}&&\ar@<-2ex>[ll]!R|{i^{*}}\ar@<2ex>[ll]!R|{i^{!}}\mathcal{B}
\ar[rr]!L|{j^{*}}&&\ar@<-2ex>[ll]!L|{j_{!}}\ar@<2ex>[ll]!L|{j_{*}}\mathcal{C}}$$
of abelian categories and additive functors such that
\begin{enumerate}
\item[(1)] ($i^{*},i_{*}$), ($i_{*},i^{!}$), ($j_{!},j^{*}$) and ($j^{*},j_{*}$) are adjoint pairs.
\item[(2)] $i_{*}$, $j_{!}$ and $j_{*}$ are fully faithful.
\item[(3)] $\Ima i_{*}=\Ker j^{*}$.
\end{enumerate}}
\end{definition}

See \cite{VFTP04C,LZQWMX11K,PC14H} for examples of recollements of abelian categories.
We list some properties of recollements (see \cite{VFTP04C,PC14H,PCSO14G,PCVJ14R}), which will be used in the sequel.

\begin{lemma}\label{lem-2.2}
Let ($\mathcal{A},\mathcal{B},\mathcal{C}$) be a recollement of abelian categories as in Definition \ref{def-2.1}.
Then we have
\begin{enumerate}
\item[(1)] $i^{*}j_{!}=0=i^{!}j_{*}$.
\item[(2)] The functors $i_{*}$ and $j^{*}$ are exact, $i^{*}$ and $j_{!}$ are right exact, and $i^{!}$ and $j_{*}$ are left exact.
\item[(3)] All the natural transformations $\xymatrix@C=15pt{i^{*}i_{*}\ar[r]&1_{\mathcal{A}},}$
$\xymatrix@C=15pt{1_{\mathcal{A}}\ar[r]&i^{!}i_{*},}$
$\xymatrix@C=15pt{1_{\mathcal{C}}\ar[r]&j^{*}j_{!}}$
and $\xymatrix@C=15pt{j^{*}j_{*}\ar[r]&1_{\mathcal{C}}}$ are natural isomorphisms.
\item[(4)] For any $B\in \mathcal{B}$, there exist exact sequences
$$\xymatrix@C=15pt{0\ar[r]&i_{*}(A)\ar[r]&j_{!}j^{*}(B)\ar[r]^-{\epsilon_{B}}&B\ar[r]&i_{*}i^{*}(B)\ar[r]&0,}$$
$$\xymatrix@C=15pt{0\ar[r]&i_{*}i^{!}(B)\ar[r]&B\ar[r]^-{\eta_{B}}&j_{*}j^{*}(B)\ar[r]&i_{*}(A')\ar[r]&0}$$
in $\mathcal{B}$ with $A,A'\in \mathcal{A}$.
\item[(5)] There exists an exact sequence of natural transformations
$$\xymatrix@C=15pt{0\ar[r]&i_{*}i^{!}j_{!}\ar[r]&j_{!}\ar[r]&
j_{*}\ar[r]&i_{*}i^{*}j_{*}\ar[r]&0.}$$
\item[(6)] If $i^{*}$ is exact, then $i^{!}j_{!}=0$; and if $i^{!}$ is exact, then $i^{*}j_{*}=0$.
\end{enumerate}
\end{lemma}

\begin{definition}
{\rm (\cite{DSE66A}) A pair of subcategories $(\mathcal{X},\mathcal{Y})$ of an abelian category $\mathcal{A}$ is called a \emph{torsion  pair}
if the following conditions are satisfied.
\begin{itemize}
\item[(1)] $\Hom_{\mathcal{A}}(\mathcal{X},\mathcal{Y})=0$; that is, $\Hom_{\mathcal{A}}(X,Y)=0$
for any $X\in\mathcal{X}$ and $Y\in\mathcal{Y}$.
\item[(2)] For any object $M\in \mathcal{A}$, there exists an exact sequence
\begin{align*}
\xymatrix@C=15pt{0\ar[r]&X\ar[r]&M\ar[r]&Y\ar[r]&0}
\end{align*}
in $\mathcal{A}$ with $X\in \mathcal{X}$ and $Y\in \mathcal{Y}$.
\end{itemize}}
\end{definition}

Let $(\mathcal{X},\mathcal{Y})$ be a torsion pair in an abelian category $\mathcal{A}$. Then we have
\begin{itemize}
\item[(1)] $\mathcal{X}$ is closed under extensions, quotient objects and coproducts.
\item[(2)] $\mathcal{Y}$ is closed under extensions, subobjects and products.
\end{itemize}
Moreover, we have
$$\mathcal{X}={^{\perp_{0}}\mathcal{Y}}:=\{M\in\mathcal{A}\mid\Hom_{\mathcal{A}}(M,\mathcal{Y})=0\},$$
$$\mathcal{Y}={\mathcal{X}^{\perp_{0}}}:=\{M\in\mathcal{A}\mid \Hom_{\mathcal{A}}(\mathcal{X},M)=0\}.$$

\begin{definition}
{\rm (\cite{HRS, BARI07H}) Let $(\mathcal{X},\mathcal{Y})$ be a torsion pair in an abelian category $\mathcal{A}$.
\begin{itemize}
\item[(1)] $(\mathcal{X},\mathcal{Y})$ is called \emph{tilting} (resp. \emph{cotilting})
if any object in $\mathcal{A}$ is isomorphic to a subobject of an object in $\mathcal{X}$
(resp. a quotient object of an object in $\mathcal{Y}$).
\item[(2)] $(\mathcal{X},\mathcal{Y})$ is called \emph{hereditary} (resp. \emph{cohereditary})
if $\mathcal{X}$ is closed under subobjects (resp. $\mathcal{Y}$ is closed under quotient objects).
\end{itemize}}
\end{definition}

\section{Torsion Pairs in a Recollement}\label{tor}

In this section, assume that ($\mathcal{A},\mathcal{B},\mathcal{C}$) is a recollement of abelian categories:
$$\xymatrix{\mathcal{A}\ar[rr]!R|{i_{*}}&&\ar@<-2ex>[ll]!R|{i^{*}}\ar@<2ex>[ll]!R|{i^{!}}\mathcal{B}
\ar[rr]!L|{j^{*}}&&\ar@<-2ex>[ll]!L|{j_{!}}\ar@<2ex>[ll]!L|{j_{*}}\mathcal{C}.}$$
We begin with the following

\begin{lemma}\label{lem-3.1}
For any $B\in\mathcal{B}$, we have
\begin{itemize}
\item[(1)] If $i^{*}$ is exact, then there exists an exact sequence
$$\xymatrix@C=15pt{0\ar[r]&j_{!}j^{*}(B)\ar[r]^-{\epsilon_{B}}&
B\ar[r]&i_{*}i^{*}(B)\ar[r]&0.}$$
\item[(2)] If $i^{!}$ is exact, then there exists an exact sequence
$$\xymatrix@C=15pt{0\ar[r]&i_{*}i^{!}(B)\ar[r]&B\ar[r]^-{\eta_{B}}&
j_{*}j^{*}(B)\ar[r]&0.}$$
\item[(3)] $i^{*}$ and $i^{!}$ are exact if and only if $i^{*}\cong i^{!}$; in this case, we have $j_{*}\cong j_{!}$.
\end{itemize}
\end{lemma}

\begin{proof}
(1) By Lemma \ref{lem-2.2}(4), it suffices to prove that $\epsilon_{B}$ is monic.
Applying $i^{!}$ to the first exact sequence in Lemma \ref{lem-2.2}(4), we get an exact sequence
$$\xymatrix@C=15pt{0\ar[r]&i^{!}i_{*}(A)\ar[r]&i^{!}j_{!}j^{*}(B).}$$
By Lemma \ref{lem-2.2}(6),  we have $i^{!}j_{!}j^{*}(B)=0$. So $A\cong i^{!}i_{*}(A)=0$ by Lemma \ref{lem-2.2}(3),
and hence $\epsilon_{B}$ is monic.

(2) It is similar to (1).

(3) If $i^{*}\cong i^{!}$, then $i^{*}$ and $i^{!}$ are exact by Lemma \ref{lem-2.2}(2). Conversely,
applying $i^{!}$ to the exact sequence in (1), we get an exact sequence of natural transformations
$$\xymatrix@C=15pt{0\ar[r]&i^{!}j_{!}j^{*}\ar[r]&i^{!}\ar[r]&i^{!}i_{*}i^{*}\ar[r]&0.}$$
By Lemma \ref{lem-2.2}(6)(3), we have $i^{!}\cong i^{!}i_{*}i^{*}\cong i^{*}$.

The isomorphism $j_{*}\cong j_{!}$ follows from Lemma \ref{lem-2.2}(5)(6).
\end{proof}

Our main result is the following

\begin{theorem}\label{Torsion pair}
Let ($\mathcal{X}',\mathcal{Y}'$) and ($\mathcal{X}'',\mathcal{Y}''$) be torsion pairs in $\mathcal{A}$ and $\mathcal{C}$
respectively, and let
\begin{align*}
\mathcal{X}&:=\{B\in\mathcal{B}\mid i^{*}(B)\in\mathcal{X}'\ \text{and}\ j^{*}(B)\in\mathcal{X}''\},\\
\mathcal{Y}&:=\{B\in\mathcal{B}\mid i^{!}(B)\in\mathcal{Y}'\ \text{and}\ j^{*}(B)\in\mathcal{Y}''\}.
\end{align*}
Then we have
\begin{itemize}
\item[(1)] ($\mathcal{X},\mathcal{Y}$) is a torsion pair in $\mathcal{B}$.
\item[(2)] $(\mathcal{X}',\mathcal{Y}')=(i^{*}(\mathcal{X}),i^{!}(\mathcal{Y}))$
and $(\mathcal{X}'',\mathcal{Y}'')=(j^{*}(\mathcal{X}),j^{*}(\mathcal{Y}))$.
\item[(3)] If ($\mathcal{X}',\mathcal{Y}'$) and ($\mathcal{X}'',\mathcal{Y}''$) are cohereditary (resp. hereditary),
and if $i^{!}$ (resp. $i^*$) is exact, then $(\mathcal{X},\mathcal{Y})$ is cohereditary (resp. hereditary).
\item[(4)] If ($\mathcal{X}',\mathcal{Y}'$) and ($\mathcal{X}'',\mathcal{Y}''$) are tilting (resp. cotilting),
and if $i^{!}$ and $j_{!}$ (resp. $i^*$ and $j_*$) are exact, then $(\mathcal{X},\mathcal{Y})$ is tilting (resp. cotilting).
\end{itemize}
\end{theorem}

\begin{proof}
(1) Let $X\in \mathcal{X}$ and $Y\in\mathcal{Y}$. Applying the functor $\Hom_{\mathcal{B}}(-,Y)$
to the following exact sequence
$$\xymatrix@C=15pt{j_{!}j^{*}(X)\ar[r]^-{\epsilon_{X}}&X\ar[r]&i_{*}i^{*}(X)\ar[r]&0}$$
in $\mathcal{B}$, we get an exact sequence
$$\xymatrix@C=15pt{0\ar[r]&\Hom_{\mathcal{B}}(i_{*}i^{*}(X),Y)\ar[r]&
\Hom_{\mathcal{B}}(X,Y)\ar[r]&\Hom_{\mathcal{B}}(j_{!}j^{*}(X),Y).}$$
By assumption, ($\mathcal{X}',\mathcal{Y}'$) and ($\mathcal{X}'',\mathcal{Y}''$) are torsion pairs
in $\mathcal{A}$ and $\mathcal{C}$ respectively.
Because $i^{*}(X)\in\mathcal{X}'$, $i^{!}(Y)\in\mathcal{Y}'$, $j^{*}(X)\in\mathcal{X}''$ and $j^{*}(Y)\in\mathcal{Y}''$,
we have
$$\Hom_{\mathcal{B}}(j_{!}j^{*}(X),Y)\cong\Hom_{\mathcal{C}}(j^{*}(X),j^{*}(Y))=0$$ and
$$\Hom_{\mathcal{B}}(i_{*}i^{*}(X),Y)\cong\Hom_{\mathcal{A}}(i^{*}(X),i^{!}(Y))=0.$$ It follows that
$\Hom_{\mathcal{B}}(X,Y)=0$ and $\Hom_{\mathcal{B}}(\mathcal{X},\mathcal{Y})=0$.

Let $B\in \mathcal{B}$. There exists an exact sequence
$$\xymatrix@C=15pt{0\ar[rr]&&i_{*}i^{!}(B)\ar[rr]&&B\ar[rr]^{\eta_{B}}
\ar@{>>}[rd]&&j_{*}j^{*}(B)\ar[rr]&&i_{*}(A')\ar[rr]&&0\\
&&&&&\Ima \eta_{B}\ar@{>->}[ru]}$$
in $\mathcal{B}$ with $A'\in\mathcal{A}$. Because $j^{*}(B)\in\mathcal{C}$ and ($\mathcal{X}'',\mathcal{Y}''$)
is a torsion pair in $\mathcal{C}$, there exists an exact sequence
$$\xymatrix@C=15pt{0\ar[r]&X''\ar[r]&j^{*}(B)\ar[r]^{h}&Y''\ar[r]&0}$$
in $\mathcal{C}$ with $X''\in \mathcal{X}''$ and $Y''\in\mathcal{Y}''$. Notice that $j_{*}$ is left exact
by Lemma \ref{lem-2.2}(2), so
$$\xymatrix@C=15pt{0\ar[r]&j_{*}(X'')\ar[r]&j_{*}j^{*}(B)\ar[r]^{j_*(h)}&j_{*}(Y'')}$$ is exact and
we have the following pullback diagram
\begin{align}\label{DV0}
\xymatrix@C=15pt{&0\ar@{-->}[d]&0\ar[d]&0\ar@{-->}[d]\\
0\ar@{-->}[r]&K\ar@{-->}[r]^{f}\ar@{-->}[d]^{g}&\ar@{-->}[r]\Ima \eta_{B}\ar[d]&\Coker f \ar@{-->}[d]\ar@{-->}[r]&0\\
0\ar[r]&j_{*}(X'')\ar@{-->}[d]\ar[r]&j_{*}j^{*}(B)\ar[r]\ar[d]&\Ima j_{*}(h)\ar[d]\ar[r]&0\\
0\ar@{-->}[r]&\Coker g\ar@{-->}[d]\ar@{-->}[r]&\ar[d]i_{*}(A')\ar@{-->}[r]&U\ar@{-->}[d]\ar@{-->}[r]&0\\
&0&0&0.}
\end{align}
Then we get the following pullback diagram
\begin{align}\label{DV2}
\xymatrix@C=15pt{&0\ar@{-->}[d]&0\ar[d]\\
&i_{*}i^{!}(B)\ar@{==}[r]\ar@{-->}[d]&i_{*}i^{!}(B)\ar[d]\\
0\ar@{-->}[r]&M\ar@{-->}[r]\ar@{-->}[d]&B\ar@{-->}[r]\ar[d]&\Coker f\ar@{==}[d]\ar@{-->}[r]&0\\
0\ar[r]&K\ar@{-->}[d]\ar[r]&\Ima \eta_{B}\ar[r]\ar[d]&\Coker f\ar[r]&0\\
&0&0.}
\end{align}
Because $i^{*}(M)\in\mathcal{A}$ and ($\mathcal{X}',\mathcal{Y}'$) is a torsion pair in $\mathcal{A}$,
there exists an exact sequence
$$\xymatrix@C=15pt{0\ar[r]&X'\ar[r]&i^{*}(M)\ar[r]&Y'\ar[r]&0}$$
in $\mathcal{A}$ with $X'\in \mathcal{X}'$ and $Y'\in\mathcal{Y}'$.
Notice that $i_{*}$ is exact by Lemma \ref{lem-2.2}(2), so
$$\xymatrix@C=15pt{0\ar[r]&i_{*}(X')\ar[r]&i_{*}i^{*}(M)\ar[r]&i_{*}(Y')\ar[r]&0}$$
is exact and we have the following pullback diagram
\begin{align}\label{DV1}
\xymatrix@C=15pt{&0\ar@{-->}[d]&0\ar[d]\\
&\Ima \epsilon_{M}\ar@{==}[r]\ar@{-->}[d]&\Ima \epsilon_{M}\ar[d]\\
0\ar@{-->}[r]&X\ar@{-->}[r]\ar@{-->}[d]&\ar@{-->}[r]M\ar[d]&i_{*}(Y')\ar@{==}[d]\ar@{-->}[r]&0\\
0\ar[r]&i_{*}(X')\ar[r]\ar@{-->}[d]&i_{*}i^{*}(M)
\ar[r]\ar[d]&i_{*}(Y')\ar[r]&0\\
&0&0,}
\end{align}
where the exactness of the middle column follows from Lemma \ref{lem-2.2}(4).
Now we get the following pushout diagram
\begin{align}\label{end}
\xymatrix@C=15pt{&0\ar[d]&0\ar@{-->}[d]\\
&X\ar@{==}[r]\ar[d]&X\ar@{-->}[d]\\
0\ar[r]&M\ar[r]\ar[d]&\ar[r]B\ar@{-->}[d]&\Coker f\ar@{==}[d]\ar[r]&0\\
0\ar@{-->}[r]&i_{*}(Y')\ar[d]\ar@{-->}[r]&Y\ar@{-->}[r]\ar@{-->}[d]&\Coker f\ar@{-->}[r]&0\\
&0&0.}
\end{align}
To get the assertion, it suffices to show $X\in\mathcal{X}$ and $Y\in\mathcal{Y}$.

Since $i^{*}j_{!}=0$ and $i^*$ is right exact by Lemma \ref{lem-2.2}(1)(2), we have $i^{*}(\Ima \epsilon_{M})=0$. Since $i^{*}$ is right exact by Lemma \ref{lem-2.2}(2),
applying the functor $i^{*}$ to the leftmost column in the diagram (\ref{DV1}) yields $i^{*}(X)\cong i^{*}i_{*}(X')\cong X'\in\mathcal{X}'$.
On the other hand, note that $j^{*}$ is exact (by Lemma \ref{lem-2.2}(2)) and $\Ima i_{*}=\Ker j^{*}$. So,
applying the functor $j^{*}$ to the bottom row in the diagram (\ref{DV0}), we have $j^{*}(\Coker g)=0=j^*(U)$; furthermore, we have
%applying the functor $j^{*}$ to the middle row in the diagram (\ref{DV1}) and the leftmost column in diagram (\ref{DV2}),
%we get $j^{*}(X)\cong j^{*}(M)\cong j^{*}(K)$. In addition, applying the functor $j^{*}$ to the bottom row and the leftmost column in the diagram (\ref{DV0}),
%we get $j^{*}(\Coker g)=0$ and $j^{*}(K)\cong j^{*}j_{*}(X'')$. So $j^{*}(X)\cong j^{*}j_{*}(X'')\cong X''\in\mathcal{X}''$, which implies $X\in \mathcal{X}$.
\begin{align*}
&\ \ \ \ j^{*}(X)\\
& \cong j^{*}(M) \ \text{(by applying} \ j^{*}\ \text{to the middle row in the diagram (\ref{DV1}))}\\
& \cong j^{*}(K) \ \text{(by applying} \ j^{*}\ \text{to the leftmost column in diagram (\ref{DV2}))}\\
& \cong j^{*}j_{*}(X'') \ \text{(by applying} \ j^{*}\ \text{to the leftmost column in the diagram (\ref{DV0}))}\\
& \cong X''\in\mathcal{X}''.
\end{align*}
It implies $X\in \mathcal{X}$.

Applying the functor $j^{*}$ to the bottom row in the diagram (\ref{end}) and the rightmost column in the diagram (\ref{DV0}),
since $j^{*}$ is exact and $\Ima i_{*}=\Ker j^{*}$, we have that $j^{*}(Y)(\cong j^{*}(\Coker f)\cong j^*(\Ima j_*(h)))$
is isomorphic to a subobject of $Y''(\cong j^{*}j_{*}(Y'')$). Because $\mathcal{Y}''$ is closed under
subobjects, it follows that $j^{*}(Y)\in\mathcal{Y}''$.
On the other hand, applying the functor $i^{!}$ to the rightmost column in the diagram (\ref{DV0}) and the bottom row
in the diagram (\ref{end}), since $i^{!}$ is left exact and $i^{!}j_{*}=0$ by Lemma \ref{lem-2.2}(1)(2), we have that $i^{!}(\Ima j_*(h))=0$ and
$i^{!}(\Coker f)=0$. So $i^{!}(Y)\cong i^{!}i_{*}(Y')\cong Y'\in\mathcal{Y}'$, and hence $Y\in \mathcal{Y}$.

(2) It is trivial that $i^{*}(\mathcal{X})\subseteq\mathcal{X}'$. For any $X'\in \mathcal{X}'$, since $i^{*}i_{*}(X')\cong X'\in\mathcal{X}'$
and $j^{*}i_{*}(X')=0\in\mathcal{X}''$, we have $i_{*}(X')\in \mathcal{X}$, and hence $X'\cong i^{*}(i_{*}(X'))\in i^{*}(\mathcal{X})$.
Thus $\mathcal{X}'\subseteq i^{*}(\mathcal{X})$. Similarly, we get $\mathcal{Y}'=i^{!}(\mathcal{Y})$,
$\mathcal{X}''=j^{*}(\mathcal{X})$ and $\mathcal{Y}''=j^{*}(\mathcal{Y})$.

(3) Assume that ($\mathcal{X}',\mathcal{Y}'$) and ($\mathcal{X}'',\mathcal{Y}''$) are cohereditary.
Then $\mathcal{Y}'$ and $\mathcal{Y}''$ are closed under quotient objects.
Let $Y\in\mathcal{Y}$ and
$$\xymatrix@C=15pt{0\ar[r]&Y'\ar[r]&Y\ar[r]&Y_{1}\ar[r]&0}$$ be an exact sequence in $\mathcal{B}$.
Since $j^{*}$ and $i^{!}$ are exact by Lemma \ref{lem-2.2}(2) and assumption, we have that $j^{*}(Y_{1})$ and $i^{!}(Y_{1})$
are isomorphic to quotient objects of $j^{*}(Y)(\in \mathcal{Y}'')$ and $i^{!}(Y)(\in \mathcal{Y}')$ respectively. So
$j^{*}(Y_1)\in \mathcal{Y}''$ and $i^{!}(Y_1)\in \mathcal{Y}'$. It implies that $Y_{1}\in \mathcal{Y}$
and ($\mathcal{X},\mathcal{Y}$) is cohereditary.

Dually, we get the assertion for the hereditary case.

(4) Assume that ($\mathcal{X}',\mathcal{Y}'$) and ($\mathcal{X}'',\mathcal{Y}''$) are tilting.
Let $B\in\mathcal{B}$. By Lemma \ref{lem-2.2}(4) and Lemma \ref{lem-3.1}(2),
there exist exact sequences
$$\xymatrix@C=15pt{0\ar[rr]&&i_{*}(A)\ar[rr]&&j_{!}j^{*}(B)
\ar[rr]^{\epsilon_{B}}\ar@{>>}[rd]&&
B\ar[rr]&&i_{*}i^{*}(B)\ar[rr]&&0,\\
&&&&&\Ima \epsilon_{B}\ar@{>->}[ru]}$$
$$\xymatrix@C=15pt{0\ar[r]&i_{*}i^{!}(B)\ar[r]
&B\ar[r]&j_{*}j^{*}(B)\ar[r]&0}$$
in $\mathcal{B}$ with $A\in\mathcal{A}$.

Since $(\mathcal{X}'',\mathcal{Y}'')$ is tilting and $j^{*}(B)\in \mathcal{C}$,
there exists a monomorphism $\xymatrix@C=15pt{0\ar[r]&j^{*}(B)\ar[r]&X''}$ in $\mathcal{C}$ with $X''\in\mathcal{X}''$.
Since $j_{!}$ is exact by assumption, we get the following exact sequence
$$\xymatrix@C=15pt{0\ar[r]&j_{!}j^{*}(B)\ar[r]&j_{!}(X'')\ar[r]&j_{!}(X''/j^{*}(B))\ar[r]&0}$$
in $\mathcal{B}$ and the following pushout diagram
\begin{align}\label{3.6}
\xymatrix@C=15pt{&&0\ar[d]&0\ar@{-->}[d]\\
0\ar[r]&i_{*}(A)\ar@{==}[d]\ar[r]&j_{!}j^{*}(B)\ar[r]\ar[d]&\Ima \epsilon_{B}\ar@{-->}[d]\ar[r]&0\\
0\ar@{-->}[r]&i_{*}(A)\ar@{-->}[r]&j_{!}(X'')\ar[d]\ar@{-->}[r]&U\ar@{-->}[d]\ar@{-->}[r]&0\\
&&j_{!}(X''/j^{*}(B))\ar[d]\ar@{==}[r]&\ar@{-->}[d]j_{!}(X''/j^{*}(B))\\
&&0&0.}\end{align}
Then we get the following pushout diagram
\begin{align}\label{3.7}
\xymatrix@C=15pt{&0\ar[d]&0\ar@{-->}[d]\\
0\ar[r]&\Ima \epsilon_{B}\ar[r]\ar[d]&\ar[r]B\ar@{-->}[d]&i_{*}i^{*}(B)\ar@{==}[d]\ar[r]&0\\
0\ar@{-->}[r]&U\ar[d]\ar@{-->}[r]&V''\ar@{-->}[r]\ar@{-->}[d]&i_{*}i^{*}(B)\ar@{-->}[r]&0\\
&j_{!}(X''/j^{*}(B))\ar[d]\ar@{==}[r]&\ar@{-->}[d]j_{!}(X''/j^{*}(B))\\
&0&0.}\end{align}

On the other hand, since $(\mathcal{X}',\mathcal{Y}')$ is tilting and $i^{!}(B)\in \mathcal{A}$, there exists a monomorphism
$\xymatrix@C=15pt{0\ar[r]&i^{!}(B)\ar[r]&X'}$ in $\mathcal{A}$ with $X'\in\mathcal{X}'$. Since $i_{*}$ is exact by Lemma \ref{lem-2.2}(2),
we get the following exact sequence
$$\xymatrix@C=15pt{0\ar[r]&i_{*}i^{!}(B)\ar[r]&i_{*}(X')\ar[r]&i_{*}(X'/i^{!}(B))\ar[r]&0}$$
in $\mathcal{B}$ and the following pushout diagram
\begin{align}\label{-2}
\xymatrix@C=15pt{&0\ar[d]&0\ar@{-->}[d]\\
0\ar[r]&i_{*}i^{!}(B)\ar[r]\ar[d]&\ar[r]B\ar@{-->}[d]&j_{*}j^{*}(B)\ar@{==}[d]\ar[r]&0\\
0\ar@{-->}[r]&i_{*}(X')\ar[d]\ar@{-->}[r]&V'\ar@{-->}[r]\ar@{-->}[d]&j_{*}j^{*}(B)\ar@{-->}[r]&0\\
&i_{*}(X'/i^{!}(B))\ar[d]\ar@{==}[r]&\ar@{-->}[d]i_{*}(X'/i^{!}(B))\\
&0&0.}
\end{align}
Then we get the following pushout diagram
\begin{align}\label{-1}
\xymatrix@C=15pt{&0\ar[d]&0\ar@{-->}[d]\\
0\ar[r]&B\ar[r]\ar[d]&\ar[r]V''\ar@{-->}[d]&j_{!}(X''/j^{*}(B))\ar@{==}[d]\ar[r]&0\\
0\ar@{-->}[r]&V'\ar[d]\ar@{-->}[r]&X\ar@{-->}[r]\ar@{-->}[d]&j_{!}(X''/j^{*}(B))\ar@{-->}[r]&0\\
&i_{*}(X'/i^{!}(B))\ar[d]\ar@{==}[r]&\ar@{-->}[d]i_{*}(X'/i^{!}(B))\\
&0&0.}
\end{align}

Since $j^{*}$ is exact (by Lemma \ref{lem-2.2}(2)) and $\Ima i_{*}=\Ker j^{*}$, we have
\begin{align*}
&\ \ \ \ j^{*}(X)\\
& \cong j^{*}(V'') \ \text{(by applying} \ j^{*}\ \text{to the middle column in the diagram (\ref{-1}))}\\
& \cong j^{*}(U) \ \text{(by applying} \ j^{*}\ \text{to the middle row in diagram (\ref{3.7}))}\\
& \cong j^{*}j_{!}(X'') \ \text{(by applying} \ j^{*}\ \text{to the middle row in the diagram (\ref{3.6}))}\\
& \cong X''\in\mathcal{X}''.
\end{align*}
Since $i^{!}$ is exact by assumption, we have $i^{*}j_{*}=0$ by Lemma \ref{lem-2.2}(6).
So, applying $i^{*}$ to the middle row in the diagram (\ref{-2}) yields that $\xymatrix@C=15pt{i^{*}i_{*}(X')\ar[r]&i^{*}(V')\ar[r]&0}$ is exact.
Since $i^{*}j_{!}=0$ by Lemma \ref{lem-2.2}(1),
applying $i^{*}$ to the middle row in the diagram (\ref{-1}) yields that
$\xymatrix@C=15pt{i^{*}(V')\ar[r]&i^{*}(X)\ar[r]&0}$ is exact.
Thus $i^{*}(X)$ is isomorphic to a quotient object of $i^{*}i_{*}(X')(\cong X'\in \mathcal{X'})$.
Notice that $\mathcal{X'}$ is closed under quotient objects, so $i^{*}(X)\in \mathcal{X'}$, and hence
$X\in \mathcal{X}$. Thus we conclude that $(\mathcal{X},\mathcal{Y})$ is tilting.

Dually, we get the assertion for the cotilting case.
\end{proof}

Recall from \cite{GR88T} that a triple of subcategories ($\mathcal{X},\mathcal{Y},\mathcal{Z}$) of an abelian category
is called a \emph{torsion torsion-free triple} (\emph{$\TTF$-triple} for short) if ($\mathcal{X},\mathcal{Y}$) and ($\mathcal{Y,\mathcal{Z}}$)
are torsion pairs. By \cite[Theorem 4.3]{PCVJ14R}, we have that ($\Ker i^{*},\Ima i_{*},\Ker i^{!}$) is a $\TTF$-triple in $\mathcal{B}$.

\begin{corollary}
Let $(\mathcal{X}',\mathcal{Y}',\mathcal{Z}')$ and $(\mathcal{X}'',\mathcal{Y}'',\mathcal{Z}'')$ are
$\TTF$-triples in $\mathcal{A}$ and $\mathcal{C}$ respectively. If $i^{*}$ and $i^{!}$ are exact,
then $(\mathcal{X},\mathcal{Y},\mathcal{Z})$ is a $\TTF$-triple in $\mathcal{B}$, where $\mathcal{X}$, $\mathcal{Y}$ are as in
Theorem \ref{Torsion pair} and
\begin{align*}
\mathcal{Z}&:=\{B\in\mathcal{B}\mid i^{*}(B)\in\mathcal{Z}'\ and\ j^{*}(B)\in\mathcal{Z}''\}.
\end{align*}
\end{corollary}

\begin{proof}
It follows from Lemma \ref{lem-3.1}(3) and Theorem \ref{Torsion pair}.
\end{proof}

To study the converse of Theorem \ref{Torsion pair}, we need the following easy observation.

\begin{lemma}\label{3.4}
If ($\mathcal{X},\mathcal{Y}$)
is a torsion pair in $\mathcal{B}$, then we have
\begin{itemize}
\item[(1)] $j_{*}j^{*}(\mathcal{Y})\subseteq\mathcal{Y}$ if and only if $j_{!}j^{*}(\mathcal{X})\subseteq\mathcal{X}$.
\item[(2)] $i_{*}i^{!}(\mathcal{Y})\subseteq\mathcal{Y}$ if and only if $i_{*}i^{*}(\mathcal{X})\subseteq\mathcal{X}$.
\end{itemize}
\end{lemma}

\begin{proof}
(1) Let $X\in\mathcal{X}$ and $Y\in\mathcal{Y}$. Since
$$\Hom_{\mathcal{B}}(X,j_{*}j^{*}(Y))\cong \Hom_{\mathcal{C}}(j^{*}(X),j^{*}(Y))\cong \Hom_{\mathcal{B}}(j_{!}j^{*}(X),Y)$$
and $\mathcal{X}={^{\perp_{0}}\mathcal{Y}}\ and\ \mathcal{Y}={\mathcal{X}^{\perp_{0}}}$, the assertion follows.

(2) It is similar to (1).
\end{proof}

The following result shows that the converse of Theorem \ref{Torsion pair} holds true under certain conditions.

\begin{theorem}\label{converse}
Let ($\mathcal{X},\mathcal{Y}$) be a torsion pair in $\mathcal{B}$. Then we have
\begin{itemize}
\item[(1)] $(i^{*}(\mathcal{X}),i^{!}(\mathcal{Y}))$ is a torsion pair in $\mathcal{A}$.
\item[(2)] $j_{*}j^{*}(\mathcal{Y})\subseteq\mathcal{Y}$ if and only if $(j^{*}(\mathcal{X}),j^{*}(\mathcal{Y}))$
is a torsion pair in $\mathcal{C}$.
\item[(3)] If $j_{*}j^{*}(\mathcal{Y})\subseteq\mathcal{Y}$, then
\begin{align*}
\mathcal{X}&=\{B\in\mathcal{B}\mid i^{*}(B)\in i^{*}(\mathcal{X})\ \text{and}\ j^{*}(B)\in j^{*}(\mathcal{X})\},\\
\mathcal{Y}&=\{B\in\mathcal{B}\mid i^{!}(B)\in i^{!}(\mathcal{Y})\ \text{and}\ j^{*}(B)\in j^{*}(\mathcal{Y})\}.
\end{align*}
\end{itemize}
\end{theorem}

\begin{proof}
(1) Let $X\in\mathcal{X}$ and $Y\in\mathcal{Y}$. Applying the functor $\Hom_{\mathcal{B}}(-,Y)$ to the exact sequence
 $$\xymatrix@C=15pt{j_{!}j^{*}(X)\ar[r]^-{\epsilon_{X}}&X\ar[r]&i_{*}i^{*}(X)\ar[r]&0}$$
in $\mathcal{B}$, we get an exact sequence
$$\xymatrix@C=15pt{0\ar[r]&\Hom_{\mathcal{B}}(i_{*}i^{*}(X),Y)\ar[r]&
\Hom_{\mathcal{B}}(X,Y)\ar[r]&\Hom_{\mathcal{B}}(j_{!}j^{*}(X),Y).}$$ Since $\Hom_{\mathcal{B}}(X,Y)=0$,
we have $\Hom_{\mathcal{B}}(i_{*}i^{*}(X),Y)=0$. It follows that $i_{*}i^{*}(X)\in{^{\perp_{0}}\mathcal{Y}}=\mathcal{X}$
and $i_{*}i^{*}(\mathcal{X})\subseteq\mathcal{X}$. So $i_{*}i^{!}(\mathcal{Y})\subseteq\mathcal{Y}$ by Lemma \ref{3.4}(2).

Let $X'\in i^{*}(\mathcal{X})$ and $Y'\in i^{!}(\mathcal{Y})$. Then there exist $X\in \mathcal{X}$ and $Y\in\mathcal{Y}$
such that $X'=i^{*}(X)$ and $Y'=i^{!}(Y)$. Because ($\mathcal{X},\mathcal{Y}$)
is a torsion pair in $\mathcal{B}$ (by assumption) and $i_{*}i^{!}(Y)\in\mathcal{Y}$, we have that
$$\Hom_{\mathcal{A}}(X',Y')=\Hom_{\mathcal{A}}(i^{*}(X),i^{!}(Y))\cong\Hom_{\mathcal{B}}(X,i_{*}i^{!}(Y))=0$$
and $\Hom_{\mathcal{A}}(i^{*}(\mathcal{X}),i^{!}(\mathcal{Y}))=0$.

Let $A\in\mathcal{A}$. Because $i_{*}(A)\in\mathcal{B}$, there exists an exact sequence
$$\xymatrix@C=15pt{0\ar[r]&X\ar[r]&i_{*}(A)\ar[r]&Y\ar[r]&0}$$
in $\mathcal{B}$ with $X\in\mathcal{X}$ and $Y\in\mathcal{Y}$.
Since $i_{*}(\mathcal{A})$ is a Serre subcategory of $\mathcal{B}$ by \cite[Proposition 2.8]{PCVJ14R},
there exist $X_{1},Y_{1}\in\mathcal{A}$ such that $X\cong i_{*}(X_{1})$ and $Y\cong i_{*}(Y_{1})$. Since
$\xymatrix@C=15pt{i_{*}: \mathcal{A}\ar[r]&i_{*}(\mathcal{A})}$ is an equivalence, we get that
$$\xymatrix@C=15pt{0\ar[r]&X_{1}\ar[r]&A\ar[r]&Y_{1}\ar[r]&0}$$
is an exact sequence in $\mathcal{A}$ with $X_{1}\cong i^{*}(i_{*}(X_{1}))\cong i^{*}(X)\in i^{*}(\mathcal{X})$
and $Y_{1}\cong i^{!}(i_{*}(Y_{1}))\cong i^{!}(Y)\in i^{!}(\mathcal{Y})$.
 %Because $i^{*}$ is right exact by Lemma \ref{lem-2.2}(2),
%we get the following exact sequence
%$$\xymatrix@C=15pt{i^{*}(X)\ar[rr]^-{i^{*}(t)}\ar@{>>}[rd]&&i^{*}i_{*}(A)(\cong A)\ar[rr]^-{i^{*}(\pi)}&&i^{*}(Y)\ar[rr]&&0.\\
%&\Ima i^{*}(t)\ \ar@{>->}[ru]}$$
%Because $i_{*}$ is exact by Lemma \ref{lem-2.2}(2) again, it follows that
%$$\xymatrix@C=15pt{i_{*}i^{*}(X)\ar[r]&i_{*}(\Ima i^{*}(t))\ar[r]&0}$$
%is also exact. Notice that $i_{*}i^{*}(X)\in\mathcal{X}$, we have $i_{*}i^{*}(\mathcal{X})\subseteq\mathcal{X}$.
%Because $\mathcal{X}$ is closed under quotient objects, we have $i_{*}(\Ima i^{*}(t))\in \mathcal{X}$, and so
%$\Ima i^{*}(t)(\cong(i^{*}i_{*}(\Ima i^{*}(t)))\in i^{*}(\mathcal{X})$. Because $i_{*}i^{*}(\mathcal{Y})\subseteq\mathcal{Y}$ by assumption,
%it follows from Lemma \ref{lem-2.2}(3) that $i^{*}(Y)(\cong i^{!}i_{*}i^{*}(Y))\in i^{!}(\mathcal{Y})$.
Thus we conclude that $(i^{*}(\mathcal{X}),i^{!}(\mathcal{Y}))$ is a torsion pair in $\mathcal{A}$.

(2) Let $j_{*}j^{*}(\mathcal{Y})\subseteq\mathcal{Y}$. For
any $X'\in j^{*}(\mathcal{X})$ and $Y'\in j^{*}(\mathcal{Y})$, there exist $X\in \mathcal{X}$ and $Y\in\mathcal{Y}$
such that $X'=j^{*}(X)$ and $Y'=j^{*}(Y)$. Because ($\mathcal{X},\mathcal{Y}$) is a torsion pair in $\mathcal{B}$, we have that
$$\Hom_{\mathcal{C}}(X',Y')=\Hom_{\mathcal{C}}(j^{*}(X),j^{*}(Y))\cong\Hom_{\mathcal{B}}(X,j_{*}j^{*}(Y))=0$$
and $\Hom_{\mathcal{C}}(j^{*}(\mathcal{X}),j^{*}(\mathcal{Y}))=0$.

Let $C\in\mathcal{C}$. Because $j_{*}(C)\in \mathcal{B}$, there exists an exact sequence
$$\xymatrix@C=15pt{0\ar[r]&X\ar[r]&j_{*}(C)\ar[r]&Y\ar[r]&0}$$
in $\mathcal{B}$ with $X\in\mathcal{X}$ and $Y\in\mathcal{Y}$. Since $j^{*}$ is exact by Lemma \ref{lem-2.2}(2), we have that
$$\xymatrix@C=15pt{0\ar[r]&j^{*}(X)\ar[r]&j^{*}j_{*}(C)(\cong C)\ar[r]&j^{*}(Y)\ar[r]&0}$$ is also exact and the assertion follows.

Conversely, if $(j^{*}(\mathcal{X}),j^{*}(\mathcal{Y}))$ is a torsion pair in $\mathcal{C}$, then we have
$$\Hom_{\mathcal{B}}(\mathcal{X},j_{*}j^{*}(\mathcal{Y}))\cong \Hom_{\mathcal{C}}(j^{*}(\mathcal{X}),j^{*}(\mathcal{Y}))=0,$$
which implies $j_{*}j^{*}(\mathcal{Y})\subseteq\mathcal{X}^{\perp_{0}}=\mathcal{Y}$.

(3) It is trivial that
$$\mathcal{X}\subseteq\{B\in\mathcal{B}\mid i^{*}(B)\in i^{*}(\mathcal{X})\ \text{and}\ j^{*}(B)\in j^{*}(\mathcal{X})\},$$
$$\mathcal{Y}\subseteq\{B\in\mathcal{B}\mid i^{!}(B)\in i^{!}(\mathcal{Y})\ \text{and}\ j^{*}(B)\in j^{*}(\mathcal{Y})\}.$$
Conversely, let $B\in\mathcal{B}$ with $i^{*}(B)\in i^{*}(\mathcal{X})$ and $j^{*}(B)\in j^{*}(\mathcal{X})$.
By Lemma \ref{lem-2.2}(4), there exists an exact sequence
$$\xymatrix@C=15pt{j_{!}j^{*}(B)
\ar[r]^-{\epsilon_{B}}&
B\ar[r]&i_{*}i^{*}(B)\ar[r]&0}$$
in $\mathcal{B}$. For any $Y\in \mathcal{Y}$, applying the functor $\Hom_{\mathcal{B}}(-,Y)$ to the above exact sequence, we get an exact sequence
$$\xymatrix@C=15pt{0\ar[r]&\Hom_{\mathcal{B}}(i_{*}i^{*}(B),Y)\ar[r]&
\Hom_{\mathcal{B}}(B,Y)\ar[r]&\Hom_{\mathcal{B}}(j_{!}j^{*}(B),Y).}$$
By (1) and (2), $(i^{*}(\mathcal{X}),i^{!}(\mathcal{Y}))$ and $(j^{*}(\mathcal{X}),j^{*}(\mathcal{Y}))$ are torsion pairs
in $\mathcal{A}$ and $\mathcal{C}$ respectively.
%Because $i^{*}(B)\in i^{*}(\mathcal{X})$ and $j^{*}(B)\in j^{*}(\mathcal{X})$,
So we have
$$\Hom_{\mathcal{B}}(j_{!}j^{*}(B),Y)\cong\Hom_{\mathcal{C}}(j^{*}(B),j^{*}(Y))=0,$$
$$\Hom_{\mathcal{B}}(i_{*}i^{*}(B),Y)\cong\Hom_{\mathcal{A}}(i^{*}(B),i^{!}(Y))=0,$$
and hence $\Hom_{\mathcal{B}}(B,Y)=0$ and $B\in{^{\perp_{0}}\mathcal{Y}}=\mathcal{X}$.
It follows that $\{B\in\mathcal{B}\mid i^{*}(B)\in i^{*}(\mathcal{X})\ \text{and}\ j^{*}(B)\in j^{*}(\mathcal{X})\}\subseteq\mathcal{X}$.
Dually, we have $\{B\in\mathcal{B}\mid i^{!}(B)\in i^{!}(\mathcal{Y})\ \text{and}\ j^{*}(B)\in j^{*}(\mathcal{Y})\}\subseteq\mathcal{Y}$.
\end{proof}

Finally, we give an example to illustrate the obtained results.

For an algebra $A$, we use $\modu A$ to denote the category of finitely generated left $A$-modules.
Let $A,B$ be artin algebras and $_AM_B$ an $(A,B)$-bimodule, and let $\Lambda={A\ {M}\choose 0\ \ B}$ be a triangular matrix algebra.
Then any module in $\modu \Lambda$ can be uniquely written as a triple ${X\choose Y}_{f}$ with $X\in\modu A$, $Y\in\modu B$
and $f\in\Hom_A(M\otimes_{B}Y,X)$ (\cite[p.76]{AMRISSO95R}).

\begin{example}
{\rm Let $A$ be a finite dimensional algebra given by the quiver $\xymatrix@C=15pt{1\ar[r]&2}$. Then $\Lambda={A\ A\choose 0\ A}$ is
a finite dimensional algebra given by the quiver
$$
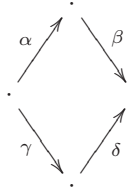
\xymatrix@C=15pt{&\cdot\ar[rd]^{\beta}\\
\cdot\ar[rd]_{\gamma}\ar[ru]^{\alpha}&&\cdot\\
&\cdot\ar[ru]_{\delta}}$$
with the relation $\beta\alpha-\delta\gamma$. The Auslander-Reiten quiver of $\Lambda$ is
$$\xymatrix@C=15pt{&P(1)\choose 0\ar[rd]&&0\choose S(2)\ar[rd]&&S(1)\choose S(1)\ar[rd]\\
S(2)\choose 0\ar[ru]\ar[rd]&&P(1)\choose S(2)\ar[ru]\ar[rd]\ar[r]&P(1)\choose P(1)\ar[r]&S(1)\choose P(1)\ar[ru]\ar[rd]&&{0\choose S(1)}.\\
&S(2)\choose S(2)\ar[ru]&&S(1)\choose 0\ar[ru]&&{0\choose P(1)}\ar[ru]}$$
By \cite[Example 2.12]{PC14H}, we have that
$$\xymatrix{\modu A\ar[rr]!R|-{i_{*}}&&\ar@<-2ex>[ll]!R|-{i^{*}}
\ar@<2ex>[ll]!R|-{i^{!}}\modu \Lambda
\ar[rr]!L|-{j^{*}}&&\ar@<-2ex>[ll]!R|-{j_{!}}\ar@<2ex>[ll]!R|-{j_{*}}
\modu A}$$
is a recollement of abelian categories, where
\begin{align*}
&i^{*}({X\choose Y}_{f})=\Coker f, & i_{*}(X)={X\choose 0},&&i^{!}({X\choose Y}_{f})=X,\\
&j_{!}(Y)={Y\choose Y}_{1}, & j^{*}({X\choose Y}_{f})=Y, &&j_{*}(Y)={0\choose Y}.
\end{align*}
\begin{itemize}
\item[(1)] Take torsion pairs $$(\mathcal{X'},\mathcal{Y'})=(\add(P(1)\oplus S(1)),\add S(2)),$$
$$(\mathcal{X''},\mathcal{Y''})=(\add S(2), \add S(1))$$ in $\modu A$. Then by Theorem \ref{Torsion pair}(1),
we get a torsion pair
$$(\mathcal{X},\mathcal{Y})=(\add ({S(2)\choose S(2)}\oplus{P(1)\choose 0}\oplus{P(1)\choose S(2)}
\oplus {0\choose S(2)}\oplus {S(1)\choose 0}),\add ({S(2)\choose 0}\oplus {0\choose S(1)}))$$ in $\modu \Lambda$.

In addition, take torsion pairs
$$(\mathcal{X'},\mathcal{Y'})=(\mathcal{X''},\mathcal{Y''})=(\add S(2), \add S(1))$$ in $\modu A$.
Then by Theorem \ref{Torsion pair}(1), we get a torsion pair
$$(\mathcal{X},\mathcal{Y})=(\add ({0\choose S(2)} \oplus {S(2)\choose S(2)}\oplus {S(2)\choose 0}),
\add ({S(1)\choose 0} \oplus {S(1)\choose S(1)}\oplus {0\choose S(1)}))$$ in $\modu \Lambda$.

\item[(2)] Take a torsion pair
$$(\mathcal{X},\mathcal{Y})=(\add({0\choose S(2)} \oplus {P(1)\choose P(1)}
\oplus {S(1)\choose 0}\oplus{S(1)\choose P(1)} \oplus {S(1)\choose S(1)} \oplus {0\choose P(1)} \oplus {0\choose S(1)}),$$
$$\add({S(2)\choose 0} \oplus {S(2)\choose S(2)}\oplus {P(1)\choose 0} \oplus {P(1)\choose S(2)}))$$ in $\modu \Lambda$.
Then by Theorem \ref{converse}(1), we have that
$$(i^{*}(\mathcal{X}),i^{!}(\mathcal{Y}))=(\add S(1),\add (S(2)\oplus P(1)))$$ is a torsion pair in $\modu A$.
Since $j_{*}j^{*}(\mathcal{Y})=\add {0\choose S(2)}\nsubseteq \mathcal{Y}$, it follows from Theorem \ref{converse}(2) that
$$(j^{*}(\mathcal{X}),j^{*}(\mathcal{Y}))=(\add(S(2)\oplus P(1)\oplus S(1)), \add S(2))$$
is not a torsion pair in $\modu A$.
%Remark that the converse of Theorem \ref{converse}(1) is not true. For example, $(i^{*}(\mathcal{X}),i^{!}(\mathcal{Y}))
%=(\add S(1),\add (S(2)\oplus P(1)))$ is a torsion pair in $\modu A$, but $i_{*}i^{*}(\mathcal{Y})=\add ({S(2)\choose 0}
%\oplus {P(1)\choose 0}\oplus {S(1))\choose 0}\nsubseteq \mathcal{Y}$.
\end{itemize}}
\end{example}

\vspace{0.5cm}
\textbf{Acknowledgement}. This work was partially supported by NSFC (No. 11571164), a Project Funded
by the Priority Academic Program Development of Jiangsu Higher Education Institutions, Postgraduate Research and
Practice Innovation Program of Jiangsu Province (Grant No. KYCX17\_0019). The first author would like to thank
Vincent Franjou and Teimuraz Pirashvili for useful suggestions. We also thank Daniel Juteau for sending us 
the paper \cite{JD09D}.

\end{document}